\documentclass{amsart}
\usepackage{amsmath}
\usepackage{amssymb}
\usepackage{enumitem} 
\usepackage{url}

\newtheorem{theorem}{Theorem}
\newtheorem{lemma}[theorem]{Lemma}
\newtheorem{corollary}[theorem]{Corollary}
\newtheorem{example}[theorem]{Example}

\newcommand{\R}{\mathbb R}
\newcommand{\C}{\mathbb C}
\newcommand{\F}{\mathbb{F}_p}
\newcommand{\FF}{\mathbb{F}}
\newcommand{\GG}{\mathbb{G}}
\newcommand{\I}{\mathcal{I}}
\renewcommand{\P}{\mathcal{P}}
\renewcommand{\L}{\mathcal{L}}
\newcommand{\eps}{\epsilon}
\newcommand{\less}{\ll}
\newcommand{\more}{\gg}

\title[An improved point-line incidence bound]{An improved point-line incidence bound over arbitrary fields}

\author{Sophie Stevens}
\address{Sophie Stevens, Department of Mathematics, University of Bristol, Bristol BS8 1TW, United Kingdom}
\email{ss1252@bristol.ac.uk}
\author{Frank de Zeeuw}
\address{Frank de Zeeuw, Department of Mathematics, \'Ecole Polytechnique F\'ed\'erale de Lausanne, Lausanne CH-1015, Switzerland}
\email{fdezeeuw@gmail.com}

\begin{document}


\begin{abstract}
We prove a new upper bound for the number of incidences between points and lines in a plane over an arbitrary field $\FF$, a problem first considered by Bourgain, Katz and Tao. 
Specifically, we show that $m$ points and $n$ lines in $\FF^2$, with $m^{7/8}<n<m^{8/7}$, determine at most $O(m^{11/15}n^{11/15})$ incidences (where, if $\FF$ has positive characteristic $p$, we assume $m^{-2}n^{13}\less p^{15}$).
This improves on the previous best known bound, due to Jones.

To obtain our bound, we first prove an optimal point-line incidence bound on Cartesian products, using a reduction to a point-plane incidence bound of Rudnev.
We then cover most of the point set with Cartesian products, and we bound the incidences on each product separately, using the bound just mentioned.

We give several applications,
to sum-product-type problems, 
an expander problem of Bourgain, 
the distinct distance problem 
and Beck's theorem.
\end{abstract}

\maketitle

\section{Introduction}\label{sec:intro}

\subsection{Notation}

We will use the following notation throughout the paper.
We let $\FF $ be an arbitrary field,
and we let $\F$ be the finite field with $p$ elements for a prime $p$.
We let $\P$ be a set of $m$ points in $\FF^2$ and $\L$ a set of $n$ lines in $\FF^2$.
We define
$$\I(\P,\L):=\left | \{(q,\ell) \in \P\times \L: q\in \ell \}\right |$$
to be the number of \emph{incidences} between $\P$ and $\L$.
We use standard asymptotic notation: $x\less y$ and $x=O(y)$ denote the existence of a constant $c>0$ such that $x\leq cy$; $x\more y$ and $x=\Omega(y)$ denote the existence of a $c>0$ such that $x\geq cy$. 
If $x\less y$ and $x\more y$, we write $x\approx y$ or $x=\Theta(y)$.

\subsection{Background}

Szemer\'edi and Trotter \cite{ST} proved the sharp upper bound 
\begin{equation}\label{eq:szemtrot}
\I(\P,\L) \less m^{2/3}n^{2/3}+m+n
\end{equation}
in the special case $\FF=\R$. 
The Szemer{\'e}di-Trotter theorem has been applied to numerous problems (see e.g. \cite{Dvir, TV}).
One famous application is the sum-product bound of Elekes \cite{Elekes}, who deduced from the Szemer\'edi-Trotter theorem that for a finite set $A\subset \R$ we have 
\begin{equation}\label{eq:elekes}
\max\{|A+A|, |A\cdot A|\} \more|A|^{5/4}.
\end{equation}
This bound has been improved using other methods (the current best bound \cite{KS, RSS} has the exponent $4/3+1/1509$, up to logarithmic factors), but Elekes's introduction of incidence geometry into additive combinatorics has proved particularly fruitful.

Various proofs of the Szemer\'edi-Trotter bound in $\R^2$ are known, but all use special properties of $\R$ that make the proofs hard to extend to other fields. 
T{\'o}th \cite{Toth} and Zahl \cite{Zahl} were successful in extending the Szemer\'edi-Trotter theorem to $\C^2$.
Proving a sharp incidence bound between points and lines over other fields, in particular finite fields, 
remains a challenging open problem.
Part of the reason for this may be that the bound \eqref{eq:szemtrot} can fail for large sets when the field $\FF$ has a finite subfield $\GG$;
in particular, if we take $\P = \GG^2$ and let $\L$ be the set of all lines in $\GG^2$, then $\I(\P,\L) \approx |\GG|^3$ but $m^{2/3}n^{2/3} \approx |\GG|^{8/3}$.

Over finite fields, the extremal situations of `very small' or `very large' sets are relatively well understood
(where size is relative to the cardinality of the field). 
For very small sets, Grosu \cite{Grosu} 
achieved the optimal bound $O(N^{4/3})$ over $\F$,
if $m,n\leq N$ and $5N <\log_2 \log_6 \log_{18} p$. 
For very large sets, Vinh \cite{Vinh} proved $\I(\P,\L)\leq mn/q+q^{1/2}\sqrt{mn}$ over any finite field $\mathbb{F}_q$.
When $m=n=N\approx q^{3/2}$, 
this bound also meets the Szemer\'edi-Trotter bound $O(N^{4/3})$.
However, for `medium-size' sets in between these extremes, these results (and their proofs) have little to say.

Over any field $\FF$, 
a simple argument (see e.g. \cite[Corollary 5.2]{BKT} or \cite[Claim 2.2]{Dvir}) using the Cauchy-Schwarz inequality gives the following combinatorial bound. 

\begin{lemma}\label{lem:combinatorial}
Let $\P$ be a set of $m$ points in $\FF^2$ and $\L$ a set of $n$ lines in $\FF^2$. 
Then
\[\I(\P,\L)\leq \min\left\{m^{1/2}n+m, mn^{1/2}+n\right\}.\]
\end{lemma}

Bourgain, Katz and Tao \cite{BKT} were the first to establish a non-trivial incidence bound in $\FF_p^2$.
They proved $\I(\P,\L)\less N^{3/2-\eps}$ for $m , n\leq  N= p^\alpha$, with $0<\alpha<2$ and $\eps=\eps(\alpha)>0$. 
They achieved this by first proving a more general but weaker version of Elekes's sum-product bound \eqref{eq:elekes}, and then, roughly speaking, running Elekes's argument in reverse.
An explicit value $\eps =1/10678$ for $\alpha=1$ was found by Helfgott and Rudnev \cite{HR},
and further improvements to $\eps$ appeared in the work of Jones \cite{Jones16,Jones12}, with the best bound summarised below.

 \begin{theorem}\label{thm:tj}\emph{(Jones \cite{Jones12})}
Let $\P$ be a set of $m$ points in $\F^2$ and $\L$ a set of $n$ lines in $\F^2$, with $m,n \leq N <p$. 
Then with $\eps= 1/662$ we have
\[ \I(\P,\L)\less N^{3/2-\eps}.\]
\end{theorem}

Jones originally stated his result only over $\F$, as his proof relied on a sum-product-type energy inequality in $\F$. 
It is a short calculation to show that his bound improves to $\eps\geq 1/326$ using recent bounds. 
Indeed, this follows by replacing \cite[Lemma 11]{Jones12} in Jones's proof by a suitable multiplicative analogue of \cite[Theorem 6]{RRS}.
Moreover, the application of \cite{RRS} swiftly extends his result to any field. 

\subsection{Main results}
Our main results are two new point-line incidence bounds over arbitrary fields. 
The first improves on Theorem \ref{thm:tj}.

\begin{theorem}
\label{thm:genszemtrot}
Let $\P$ be a set of $m$ points in $\FF^2$ and $\L$ a set of $n$ lines in $\FF^2$,
with $m^{7/8} < n < m^{8/7}$.
If $\FF$ has positive characteristic $p$,
assume $m^{-2}n^{13}\less p^{15}$.
Then 
$$\I(\P,\L)\less m^{11/15}n^{11/15}.$$
\end{theorem}

When $m = n=N$, 
Theorem \ref{thm:genszemtrot} improves the $\eps$ in Theorem~\ref{thm:tj} from $1/662$ to $1/30$, 
 it extends the condition in positive characteristic to $N\less p^{15/11}$, 
 and it has the further advantage of being sensitive to the relative sizes of the point set and line set.
To compare it with the bound of Vinh \cite{Vinh}, assume $\FF = \F$ and $m=n=N$; 
then Theorem \ref{thm:genszemtrot} is better for $N\less p^{15/14}$.
We remark that  Theorem \ref{thm:genszemtrot}, and most results in this paper, are meaningful only if the characteristic $p$ is large, since if $p=O(1)$ then $m,n =O(1)$. 
We could state the bound in the Szemer\'edi-Trotter-like form $\I(\P,\L)\less m^{11/15}n^{11/15} +m+n$, 
but this might be misleading, since outside the range $m^{7/8} < n < m^{8/7}$ the bounds from Lemma \ref{lem:combinatorial} are better.
We summarise the situation in Table \ref{table:bounds}.

\setlength{\tabcolsep}{10pt}
\renewcommand{\arraystretch}{1.3}
\begin{table}[ht]
\begin{center}
\begin{tabular}{ |l|l| } 
 \hline
 \multicolumn{1}{|c|}{Range of $n$} & \multicolumn{1}{|c|}{Best bound}  \\ 
 \hline\hline
 \hspace{35pt}$n < m^{1/2}$ & $\less m$ \\ 
 \hline
 $m^{1/2}< n < m^{7/8}$ & $\less m^{1/2} n$  \\ 
 \hline
 $m^{7/8} < n < m^{8/7}$ & $\less m^{11/15}n^{11/15}$  \\ 
 \hline
 $m^{8/7}<n < m^2$ & $\less m n^{1/2}$  \\ 
 \hline
 \hspace{8pt}$m^2 < n  $ & $\less n$  \\ 
 \hline
\end{tabular}
\vspace{10pt}
\caption{Overview of best known upper bounds on $\I(\P,\L)$}
\label{table:bounds}
\end{center}
\end{table}
\vspace{-15pt}

Bourgain, Katz and Tao \cite{BKT} first made the observation that if there are many incidences, a large grid-like structure must exist in the point set,
and this idea was later refined by Jones \cite{Jones12}.
Our proof of Theorem \ref{thm:genszemtrot} is based on the same idea: We cover the point set by grid-like structures that are projectively equivalent to Cartesian products,
and then we bound the incidences on each Cartesian product using the following new incidence bound.

\begin{theorem}
\label{thm:cartszemtrot}
Let $A, B\subset \FF$ be sets with $|A|=a$, $|B|=b$, $a\leq b$ and $ab^2\leq n^3$.
Let $\L$ be a finite set of lines of size $n$.
If $\FF$ has positive characteristic $p$,
assume $a n \less p^2$.
Then 
\[ \I(A\times B,\L) \less a^{3/4}b^{1/2}n^{3/4} +n . \] 
\end{theorem}

When $a=b = m^{1/2}$, the bound in Theorem \ref{thm:cartszemtrot} becomes $O(m^{5/8}n^{3/4}+n)$, 
which improves a similar statement in Aksoy Yazici et al. \cite{AMRS}
(for comparison, the bound in \cite{AMRS} is $O(m^{3/4}n^{2/3}+n)$).
Note that our bound has the somewhat unusual property that it is better for uneven products, i.e., for products $A\times B$ of a fixed size $ab$, 
the bound gets better as $a/b$ decreases.

We use Theorem \ref{thm:cartszemtrot} to prove Theorem \ref{thm:genszemtrot}, 
but it is also interesting in its own right.
Indeed, in some applications of point-line incidence bounds, 
the point set is a Cartesian product;
see Sections \ref{sec:appscart} and \ref{sec:exppoly} for several examples.
Theorem~\ref{thm:cartszemtrot} can also be used to obtain a near-sharp estimate on the number of collinear quadruples in a planar point. 
We refer the reader to Petridis \cite{petridis} for this estimate, and a subsequent work of Murphy et al.  \cite{mprnrs} for a wealth of applications.

In general, Theorem \ref{thm:cartszemtrot} is quantitatively weaker than the Szemer\'edi-Trotter bound \eqref{eq:szemtrot}.
If, for instance, we consider Cartesian products $\P=A\times B$ with $|A|=|B|$ and $|\L|=|\P|=|A|^2$, 
then Theorem~\ref{thm:cartszemtrot} gives the bound $\I(\P,\L)\ll |A|^{11/4}$,
whereas \eqref{eq:szemtrot} gives $\I(\P,\L)\ll |A|^{8/3}$ (over $\R$).
Nevertheless, 
Theorem \ref{thm:cartszemtrot} is optimal for certain sets of points and lines, as the following construction of Elekes \cite{E02} demonstrates. 
\begin{example}\label{ex:elekes}
Let $a,c$ be integers.
If $\FF$ has positive characteric $p$, assume that $2ac<p$.
We define a point set by
$$\P=\{(i,j): i =1,\dots ,a, ~j =1,\dots ,2ac \}$$
and a line set by
$$\L =\{y=sx+t: s=1,\dots, c,~ t =1,\dots, ac\}\,.$$
Using the condition $2ac<p$, we have $|\P| =2a^2c$ and $|\L|=ac^2$.
Every line in $\L$ contains $a$ points of $\P$, so $\I(\P,\L)= a^2c^2$.
Theorem~\ref{thm:cartszemtrot} yields the matching bound
$$ \I(\P,\L)\ll a^{3/4}(2ac)^{1/2}(ac^2)^{3/4}\ll a^2c^2\,.$$
\end{example}

Note that Example \ref{ex:elekes} works for any choice of the sizes $|\P|$ and $|\L|$; roughly speaking, we can set $a\approx (|\P|^2/|\L|)^{1/3}$ and $c\approx (|\L|^2/|\P|)^{1/3}$.
However, the choice of $|\P|$ and $|\L|$ may force the Cartesian product $\P = A\times B$ to be rather uneven, in the sense that $|B|$ is much larger than $|A|$.
Given that the Szemer\'edi-Trotter bound \eqref{eq:szemtrot} is better for more balanced products over $\R$, 
it appears likely that Theorem~\ref{thm:cartszemtrot} is not optimal for all choices of $|A|$, $|B|$ and $|L|$.

\subsection{Applications}
Theorem \ref{thm:genszemtrot} and Theorem \ref{thm:cartszemtrot} lead to several improved bounds in well-known problems.
We summarize the applications here, and in Section \ref{sec:apps} we will provide some background and full proofs for most of the results. 

First of all, 
we use Theorem \ref{thm:cartszemtrot} to deduce the sum-product bounds
\[\max\{|A+A|,|A\cdot A|\} \more |A|^{6/5} \hspace{10pt}\text{and} \hspace{10pt} |A\cdot (A+1)| \more |A|^{6/5}\]
for a finite set $A\subset \FF$ (assuming $|A|\less p^{5/8}$ in positive characteristic $p$).
The first bound was originally obtained  
by Roche-Newton, Rudnev and Shkredov \cite{RRS}, but here we recover it using Elekes's original argument for \eqref{eq:elekes}.
The second inequality improves slightly on a result 
from \cite{RSS}.
We also reprove (assuming $|A|\less p^{2/3}$) the bounds 
\[|A+AA|\more |A|^{3/2}\hspace{10pt}\text{and} \hspace{10pt}
|A(A+A)|\more |A|^{3/2},\]
which were first proved in  \cite{RRS} and \cite{AMRS} respectively.
Next, we use Theorem \ref{thm:cartszemtrot} to prove that the polynomial $f(x,y) = x^2+xy$ satisfies 
\[|f(A,B)| \more N^{5/4} \]
for any finite sets $A,B\subset \FF$ with $|A|=|B|=N$ (assuming $N\less p^{2/3}$ in positive characteristic $p$).
This improves the exponent for a result of Bourgain \cite{Bourgain05}.

We give two geometric applications of Theorem \ref{thm:genszemtrot}.
For a set $\P$ in $\FF^2$, 
write $\Delta(\P) = \{(p_x-q_x)^2+(p_y-q_y)^2: p,q\in \P\}$ for the set of squared `Euclidean distances' determined by $\P$.
We prove that
\[ |\Delta(\P)| \more m^{8/15},\]
assuming $m\less p^{15/11}$ if $\FF$ has positive characteristic $p$, 
and assuming $\Delta(\P)\neq \{0\}$ if $-1$ is a square in $\FF$.
This improves on a result of Bourgain, Katz and Tao \cite{BKT}.
Next, we prove the following analogue of a theorem of Beck, improving on Jones \cite{Jones12}.
For a finite set $\P$ of $m$ points in $\FF^2$ (if $\FF$ has positive characteristic $p$, assume $m\less p^{7/6}$),
either $\P$ has $\Omega(m)$ points on a line, or $\P$ determines $\Omega(m^{8/7})$ lines.

Finally, we record in this note that our Theorem \ref{thm:genszemtrot} can be plugged into a result of Lewko \cite{Lewko} to give an improved restriction estimate for the paraboloid in $\F^3$. A contextual discussion of this problem as well as definitions of forthcoming notation are found in \cite{Lewko,MT}.
Specifically, 
combining \cite[Theorem 2]{Lewko} with Theorem \ref{thm:genszemtrot} shows (assuming $-1$ is not a square in $\F$) that the statement $\mathcal{R}^*(2\to 68/19+\eps)\less 1$ holds for all $\eps>0$ for the paraboloid $S$ defined by $z=x^2+y^2$ in $\F^3$, 
which means that 
\[\|(gd\sigma)^\vee\|_{L^{68/19+\eps}(S,d\sigma)} \less \|g\|_{L^{2}(\F^3,d\sigma)}. \]
For comparison, the exponent $68/19 = 18/5 - 2/95 \leq 3.579$ improves on Lewko's exponent $745/207 = 18/5 - 1/1035 \geq 3.599$,
which in turn improved on the exponent $18/5$ obtained by Mockenhaupt and Tao \cite{MT}. We refer the interested reader to \cite{Lewko} for a thorough treatment of this problem.

\subsection{Discussion}
We will deduce Theorem \ref{thm:genszemtrot} from Theorem \ref{thm:cartszemtrot}, and we will derive Theorem \ref{thm:cartszemtrot} from the following point-plane incidence bound of Rudnev \cite{Rudnev}.

\begin{theorem}[Rudnev]\label{thm:rudnev}
Let $\mathcal{R}$ be a set of $r$ points in $\FF^3$ and let $\mathcal{S}$ be a set of $s$ planes in $\mathbb{F}^3$, with $r\leq s$. 
If $\mathbb{F}$ has positive characteristic $p$, suppose that  
$r\less p^2$. 
Let $k$ be the maximum number of collinear points in $\mathcal{R}$. 
Then
\[ \I(\mathcal{R},\mathcal{S})\less r^{1/2}s +k s.\] 
\end{theorem}

This bound is tight if $k\geq r^{1/2}$; if $k$ of the points are on a line, and all $s$ planes contain that line, then there are $ks$ incidences.
Over $\R$, there are known to be better bounds for small values of $k$.

Theorem \ref{thm:rudnev} is based on a ground-breaking result of Guth and Katz 
\cite{GK},
which states that a set of $N$ lines in $\FF^3$, 
with no $N^{1/2}$ lines on a quadric surface,
 determines $O(N^{3/2})$ intersection points.
This result played a role in the resolution of the Erd\H os distinct distance problem in $\R^2$ in \cite{GK}.
The proof in \cite{GK} uses interpolation to capture the lines in an algebraic surface of degree $N^{1/2}$, and then analyses the intersections between the lines in each irreducible component of that surface.
The result in \cite{GK} was stated over $\R$, but essentially the same proof works over any field,
if in positive characteristic one adds the restriction $N\less p^2$ (see Rudnev \cite{Rudnev} and Koll\'ar \cite[Corollary 40]{Kollar}). A recent work by the second author \cite{deZeeuw} shortened Rudnev's proof, removing some of the technicalities. 

To summarise, a proof from scratch of Theorem \ref{thm:genszemtrot} would proceed as follows: 
We cover $\P$ by grids, 
for each grid we convert the point-line incidences to point-plane incidences,
these are then converted to line-line intersection points in space (as in \cite{Rudnev}), and these are finally bounded using the algebraic techniques in \cite{GK}.
This gives a nice picture of the connections between the different types of incidence bounds.

Previous approaches to incidence problems over finite fields relied on techniques from additive combinatorics, and in particular the Balog-Szemer\'edi-Gowers (BSG) theorem \cite{BSG}. 
A number of recent papers, including \cite{RRS, AMRS, RSS}, 
have successfully replaced the traditional application of BSG by more geometric arguments, leading to several quantitative improvements.
Our proof of Theorem \ref{thm:genszemtrot} is another example where BSG is replaced by the geometry inherent within the problem.

An obvious question is how to further improve the bounds in Theorem \ref{thm:genszemtrot}
and Theorem \ref{thm:cartszemtrot}.
We note that, even if the main term in the bound of Theorem \ref{thm:cartszemtrot} were improved to $O(m^{2/3}n^{2/3})$, our proof would not lead to the same bound in Theorem \ref{thm:genszemtrot} (the result would be $O(m^{8/11}n^{8/11})$).
Another interesting open problem, first posed by Bourgain \cite{Bourgain05}, is whether similar bounds can be obtained for non-linear objects, like circles, conics, or other algebraic curves.
Over $\R$ and $\C$ such bounds are known (see e.g. \cite{Dvir,SSZ}), 
and over $\F$ Bourgain \cite{Bourgain12} proved an incidence bound for hyperbolas.

\subsection{Organisation}
The rest of the paper is structured as follows.
In Section \ref{sec:cartszemtrot} we prove Theorem \ref{thm:cartszemtrot},
while in Section \ref{sec:grids} and Section \ref{sec:genszemtrot} we use Theorem \ref{thm:cartszemtrot} to prove Theorem \ref{thm:genszemtrot}.
Finally, in Section \ref{sec:apps}, we work through some of the applications of these new incidence bounds.

 
\section{A point-line incidence bound on Cartesian products}\label{sec:cartszemtrot}

As a first step towards Theorem \ref{thm:genszemtrot}, we prove a stronger point-line incidence bound when the point set is a Cartesian product.
The idea is to think of point-line incidences as solutions $(x,y,s,t)$ of the equation $xs+t=y$, where $(x,y)$ is a point and $(s,t)$ represents a line. 
The number of such solutions can be related to the number of solutions of $xs+t=x's'+t'$.
Such a bilinear equation in six variables can then be turned into a point-plane incidence problem, to which Theorem \ref{thm:rudnev} applies.
The fact that the point set is a Cartesian product is crucial, because it allows us to `split' the variable $x$ from the variable $y$.

The key to relating the solutions of $xs+t=y$ to the solutions of $xs+t=x's'+t'$ is the Cauchy-Schwarz inequality.
Although it is fairly standard, 
let us explicitly formalise this connection, 
since we will use it again in Section \ref{sec:apps}.

\begin{lemma}\label{lem:cauchyschwarz}
Let $X$ and $Y$ be finite sets, and let $\varphi:X\to Z$ be a function, where $Y\subset Z$.
Then
\[|\{(x,y)\in X\times Y: \varphi(x)=y\}| 
\leq |Y|^{1/2} \cdot |\{(x,x')\in X\times X: \varphi(x) = \varphi(x')\}|^{1/2}.\]
\end{lemma}
\begin{proof}
Set $X_z := \{x\in X: \varphi(x) = z\}$ for any $z\in Z$.
By the Cauchy-Schwarz inequality, we have
\[|\{(x,y)\in X\times Y: \varphi(x)=y\}| =\sum_{y\in Y} |X_y| \leq |Y|^{1/2}\sum_{y\in Y}|X_y|^2. \]
Combining this with
\[ \sum_{y\in Y} |X_y|^2 \leq 
\sum_{z\in Z} |X_z|^2 \leq 
|\{(x,x')\in X\times X: \varphi(x) = \varphi(x')\}|\]
proves the lemma.
\end{proof}

\begin{proof}[Proof of Theorem \ref{thm:cartszemtrot}]
Recall that $|\L| = n$, $|A|=a$, $|B|=b$, $a\leq b$ and $ab^2\leq n^3$.
First we show that, by modifying $\L$, we can assume the following three properties.
\begin{itemize}
\item {\bf There are no vertical lines in $\L$.}\\
We remove all vertical lines from $\L$; they together contribute at most $ab$ incidences, 
and the assumption $ab^2\leq n^3$ implies $ab\leq a^{3/4}b^{1/2}n^{3/4}$.
\item {\bf We have $b^2\leq an$.}\\
Given that there are no vertical lines,
we have the bound $\I(A\times B,\L)\leq an$, since each line  from $\L$ intersects each of the $a$ vertical lines covering $A\times B$ at most once.
If $b^2> an$, 
then we get $\I(A\times B,\L)\leq an\leq a^{3/4}b^{1/2}n^{3/4}$.
\item {\bf At most $a^{1/2}n^{1/2}$ lines of $\L$ are concurrent or parallel.}\\
We iteratively remove any pencil (a \emph{pencil} is a set of concurrent lines) of more than $a^{1/2}n^{1/2}$ concurrent or parallel lines. 
Let $n_i$ be the number of lines in the $i$-th pencil that we remove (not counting those that were removed earlier).
Then the $i$-th pencil is involved in at most $ab +n_i$ incidences.
We need at most $n/(a^{1/2}n^{1/2}) = a^{-1/2}n^{1/2}$ steps to remove all such pencils.
In pruning the line set in this manner, we discount at most $a^{-1/2}n^{1/2}\cdot ab + \sum n_i\leq a^{3/4}b^{1/2}n^{3/4}+n$ incidences, where we used the assumption $b^2\leq an$.
\end{itemize}

Since $\L$ has no vertical lines, 
the affine dual $\L^*:= \{ (c,d)\in \FF^2 : y=cx+d\in \L\}$ of $\L$ is well-defined.
Then we have\footnote{We abuse notation by denoting an element in $A\times B\times \L^*$ by $(a,b,c,d)$ instead of $(a,b, (c,d))$.}
\[\I(\P,\L) = |\{ (x,y,s,t)\in A\times B\times \L^* : xs+t=y \}|. \]
If we set
 \[E := \{(x,s,t,x',s',t')\in (A\times \L^*)^2 : xs+t = x's'+t'\}, \]
then Lemma \ref{lem:cauchyschwarz} (with $X = A\times \L^*$, $Y = B$, $Z = \FF$ and $\varphi(x,s,t) = xs+t$) gives
\begin{equation}\label{eq:csbound}
\I(A\times B,\L) 
 =b^{1/2} |E|^{1/2}.
 \end{equation}

We bound $|E|$ using the point-plane incidence bound in Theorem \ref{thm:rudnev}.
Define a point set and a plane set by
\[ \mathcal{R} := \{(x,s',t')\in A\times \L^*\},~~~~~~
\mathcal{S} := \{x s + t = x' s'+ t':(x',s,t)\in A\times \L^*\}.\] 
We have $|\mathcal{R}| = |\mathcal{S}| = an$ and $|E| = \I(\mathcal{R},\mathcal{S})$.
 
To apply Theorem \ref{thm:rudnev} we need to check its conditions. 
The condition that the number of points is $O(p^2)$ follows from the assumption that $an \less p^2$.
The condition that there are at most as many points as planes clearly holds, since $|\mathcal{R}|= |\mathcal{S}|$.
Because of the product structure of $\mathcal{R} = A\times \L^*$, the maximum number of collinear points in $\mathcal{R}$ is bounded by the maximum of $a$ 
and the maximum number of collinear points in $\L^*$.
The former is bounded by $a^{1/2}n^{1/2}$, using the fact that $a\leq n$, which follows from $a\leq b$ and $ab^2\leq n^3$.
The latter equals the maximum number of concurrent lines in $\L$, 
which by our earlier assumption is also bounded by $a^{1/2}n^{1/2}$.

Therefore, we can apply Theorem \ref{thm:rudnev} with $k = a^{1/2}n^{1/2}$ to obtain
\[\I(\mathcal{R},\mathcal{S})
\less |\mathcal{R}|^{1/2}|\mathcal{S}| + k |\mathcal{S}| 
\less a^{3/2}n^{3/2}.\]
Combining this with \eqref{eq:csbound} and $|E| = \I(\mathcal{R},\mathcal{S})$ gives
\[\I(A\times B,\L) 
\less  b^{1/2}|E|^{1/2}
\less a^{3/4}b^{1/2}n^{3/4}\,,\]
proving the theorem.
\end{proof}

Note that the term $n$ in the bound of Theorem \ref{thm:cartszemtrot} comes only from the step in the proof where we ensured that at most $a^{1/2}n^{1/2}$ lines of $\L$ are concurrent or parallel.
Also observe that we could have stated the bound in the slightly stronger form $Ca^{3/4}b^{1/2}n^{3/4} +n$ for a constant $C$.
 
\section{Finding a Cartesian product}\label{sec:grids}
In order to apply Theorem~\ref{thm:cartszemtrot} to an unstructured point set, 
we require a means to find large grids in a point set with many incidences. 
This approach was first taken in the original incidence bound over $\F$ in \cite{BKT}, 
where the authors showed that if a point set has many incidences, then a large subset of the points can be captured inside the intersection of two relatively small pencils.
Then they used the fact that the set of intersection points of two pencils is projectively equivalent to a Cartesian product.
This approach was quantitatively refined by Jones in \cite{Jones12},
who showed that, after carefully `regularising' the points, 
$\P$ can be efficiently partitioned into a number of subsets, 
each of which is covered by two relatively small pencils.

Our approach is also based on the fact that if a set is `regular' in the sense that each point lies on a similar number of lines, then there are two pencils whose intersection covers many points of $\P$.
This fact is captured in Lemma \ref{lem:twopencils} below.
This lemma is a quantitative version of Proposition 4 of \cite{Jones12}.
We avoid asymptotic notation in this section,
because in the next section we will apply Lemma \ref{lem:twopencils} inside an induction, 
where we have to be careful with the dependence of the constants.
We denote by $\overline{pq}$ the line in $\FF^2$ containing the points $p, q \in \FF^2$.

\begin{lemma}\label{lem:twopencils}
The following holds for any constants $c_2>c_1>0$.

Let $\P$ be a set of $m$ points and $\L$ a set of $n$ lines,
such that between $c_1 K$ and $c_2 K$ lines of $\L$ pass through each point in $\P$. 
Assume $K\geq 4n/(c_1m)$, $K\geq 8/c_1$ and $K^3\geq 2^6n^2/(c_1^3m)$.

Then there are distinct points $p_1,q_1\in\P$ and a set $G \subseteq \P\backslash\overline{p_1q_1}$ of cardinality $|G|\nobreak\geq\nobreak c_1^4 K^4m/(2^9n^2)$, such that 
$G$ is covered by at most $c_2K$ lines from $\L$ through $p_1$, 
and by at most $c_2K$ lines from $\L$ through $q_1$.
\end{lemma}
\begin{proof}
Let
\[ \L_1:=\left\{\ell\in\L: |\ell\cap \P|\geq
\I(\P,\L)/(2n)
\right\}.\]
Then we have $\I(\P,\L_1)\geq\I(\P,\L)/2$,
since the set of lines not contained in $\L_1$ contribute fewer than $n\cdot \I(\P,\L)/(2n) = \I(\P,\L)/2$ incidences to $\I(\P,\L)$.
Let $p_1\in \P$ be a point incident to at least $\I(\P,\L_1)/(2m)$ lines in $\L_1$. 
Such a point exists since the set of points that are incident to fewer than $\I(\P,\L_1)/2m$ lines contribute fewer than $m\cdot \I(\P,\L_1)/(2m) = \I(\P,\L_1)/2$ incidences to $\I(\P,\L_1)$.

Note that the assumptions of the lemma imply $\I(\P,\L)\geq c_1Km$,
so we have $\I(\P,\L)/(2n)\geq c_1 Km/(2n)$ and $\I(\P,\L_1)/(2m)\geq (\I(\P,\L)/2)/(2m)\geq c_1K/4$.
Thus the point $p_1$ is incident to at least $c_1K/4$ lines from $\L_1$, and each line in $\L_1$ is incident to at least $(c_1 Km/(2n))-1$ points in $\P\backslash\{p_1\}$. 
It follows that  
\[\mathcal{Q} := \{q\in\P\backslash\{p_1\} :\overline{p_1q}\in\L\}\]
satisfies
\begin{equation}\label{eq:Qbound}
|\mathcal{Q}| \geq \frac{c_1 K }{4}\left(\frac{c_1 Km}{2n}-1\right) \geq \frac{c_1^2 K^2 m}{2^4 n},
\end{equation}
where in the last inequality we used the assumption $K\geq 4n/(c_1m)$.

The points in $\mathcal{Q}$ still have the property that between $c_1K$ and $c_2K$ lines of $\L$ pass through them, so we can repeat the argument above,
with $\mathcal{Q}$ in the role of $\P$, 
and the same line set $\L$.
We let 
\[ \L_2:=\left\{\ell\in\L: |\ell\cap \mathcal{Q}|\geq
\I(\mathcal{Q},\L)/(2n)
\right\}.\]
As above, 
we have $\I(\mathcal{Q},\L_2)\geq \I(\mathcal{Q},\L)/2$,
and there is a point $q_1\in \mathcal{Q}$ that is incident to at least $\I(\mathcal{Q},\L_2)/(2|\mathcal{Q}|)\geq c_1K/4$ lines in $\L_2$. 
Thus $q_1$ is incident to at least $(c_1K/4)-1$ lines from $\L_2$ other than the line $\overline{p_1q_1}$, 
and each line in $\L_2$ is incident to at least $\I(\mathcal{Q},\L)/(2n)\geq c_1 K|\mathcal{Q}|/(2n)$ points in $\P$. 
Thus the set
\[\mathcal{R} := \{q\in\mathcal{Q}\backslash\overline{p_1q_1} :\overline{q_1q}\in\L_2\}\] satisfies
\[|\mathcal{R}| \geq \left(\frac{c_1 K }{4} -1\right)\left(\frac{c_1 K|\mathcal{Q}|}{2n}-1\right)\geq \frac{c_1 K }{8}\cdot \frac{c_1 K|\mathcal{Q}|}{4n} = \frac{c_1^2 K^2 |\mathcal{Q}|}{2^5 n}\geq \frac{c_1^4 K^4 m}{2^9 n^2},\]
where in the second inequality we used $K\geq 8/c_1$ in the first factor,
 and both \eqref{eq:Qbound} and $K^3\geq 2^6n^2/(c_1^3m)$ in the second factor,
while in the last inequality we used \eqref{eq:Qbound}.

As $p_1$ is incident to at most $c_2 K$ lines, 
$\mathcal{Q}$ is covered by at most $c_2K$ lines from $\L$ that pass through $p_1$, 
and therefore so is $\mathcal{R}\subset \mathcal{Q}$.
Similarly, $\mathcal{R}$ is covered by at most $c_2K$ lines from $\L$ that pass through $q_1$.
Therefore,
we can choose the point set $G$ as a subset of $ \mathcal{R}$ with $|G|\geq c_1^4 K^4m/(2^9n^2)$.
This concludes the proof.
\end{proof}

\section{Proof of Theorem~\ref{thm:genszemtrot}}\label{sec:genszemtrot}

We will prove that there exists a constant $C$
such that, for all $\P$ and $\L$  with $n^{7/8}<m<n^{8/7}$, we have
$$\I(\P,\L)< Cm^{11/15}n^{11/15}.$$
We do this by induction, keeping $n$ fixed and varying $m$. 
The inductive hypothesis is that for any point set $\mathcal{P'}$ satisfying $|\mathcal{P'}|=m'$, 
where $n^{7/8}<m'<m$, 
we have $\mathcal{I}(\mathcal{P'},\mathcal{L})<C(m')^{11/15}n^{11/15}$.
The base case of the induction is any $m$ such that $n^{4/11}< m < n^{7/8}$, for which Lemma \ref{lem:combinatorial} gives
\[ \I(\P,\L)\leq mn^{1/2}+n \leq 2 m^{11/15}n^{11/15}.\]

We argue by contradiction;
we will suppose that $\I(\P,\L)=Cm^{11/15}n^{11/15}$,
and we show, using the inductive hypothesis and the assumption $n^{7/8}<m<n^{8/7}$, that for a sufficiently large choice of $C$, independent of $m$ and $n$, a contradiction occurs.
We will work with explicit constants in the proof; we choose the constants for ease of comprehension, and we make no attempt to optimise them. 

As said, we suppose that $n^{7/8}<m<n^{8/7}$ and $I:=\I(\P,\L)=Cm^{11/15}n^{11/15}$.
Set $K := I/m$.
We introduce two subsets of $\P$: 
$$D:=\{p\in \P:\text{ there are at most } 2^{-11}K
\text{ lines through } p\}$$
and 
$$E:=\{p\in \P:\text{ there are at least } 2^{15}K
\text{ lines through } p\}.$$
One can think of $D$ as the set of points with a dearth of incidences, and $E$ as the set of points with an excess of incidences.

It is evident that $D$ contributes at most $2^{-11}Km = 2^{-11}I$ incidences to $I$.
Similarly, we have the estimate $I\geq \I(E,\L) \geq 2^{15} K|E|$,
which implies $|E|\leq 2^{-15}m$. 
By induction we have
\[\I(E,\L) < C\left(2^{-15}m\right)^{11/15}n^{11/15}
<2^{-11}I.\]
So $E$ also contributes at most $2^{-11}I$ incidences to $I$.

Let $A:=\P\backslash (E\cup D)$ be the remaining points.
By definition of $D$ and $E$, every point in $A$ is incident to at least $c_1K$ and at most $c_2K$ lines of $\L$.
From the previous paragraph, we know that $A$ contributes at least $\left(1 - 2\cdot 2^{-11}\right)I$ incidences to $I$.

We repeatedly use Lemma~\ref{lem:twopencils} with $c_1 = 2^{-11}$ and  $c_2 = 2^{15}$,
to get the following sequence of grid-like subsets.
Let $A_1:=A$. 
We iteratively choose $G_i\subset A_i$ as in Lemma~\ref{lem:twopencils},
so there exist distinct points $p_i, q_i$ such that 
$G_i$ is covered by at most $2^{15}K$ lines from $\L$ through $p_i$, 
and by at most $2^{15}K$ lines from $\L$ through $q_i$.
Then we set $A_{i+1}=A_i\backslash G_i$ and repeat.
We terminate this process at the $s$-th step when $|A_{s+1}|\leq 2^{-15}m$ (allowing for the possibility that $s=0$, which happens if $|A| \leq 2^{-15}m$, and the process is empty).
This results in a sequence $A_1 \supseteq A_2\supseteq \dots \supseteq A_{s+1}$,
with 
\[|G_i|\geq \frac{c_1^4K^4 |A_i|}{2^9n^2} 
\geq \frac{(2^{-11})^4K^4 (2^{-15}m)}{2^9n^2}
\geq \frac{K^4 m}{2^{68}n^2}\]
As the $G_i$ are disjoint by construction, the process terminates after at most 
\[s\leq \frac{m}{\min_i\{|G_i|\}} \leq \frac{2^{68} n^2}{K^4}\] 
steps.
It is a straightforward calculation to show that throughout the process, the conditions 
$K\geq 4n/(c_1|A_i|)$, $K\geq 8/c_1$ and $K^3\geq 2^6n^2/(c_1^3|A_i|)$ 
 of Lemma \ref{lem:twopencils} hold if $C$ is chosen sufficiently large.

We may apply the inductive assumption to bound 
\[\I(A_{s+1},\L)<C(2^{-15}m)^{11/15}n^{11/15} = 2^{-11} I.\] 
Thus the subsets $G_1,\dots ,G_s$ contribute at least $\left(1 - 3\cdot 2^{-11}\right)I\geq I/2$ incidences, 
and in particular we have
\begin{equation}\label{eq:IlessIAL}
I \leq  2 \sum_{i=1}^s \I(G_i, \L).
\end{equation}

We now show that each $G_i$ is projectively equivalent to a Cartesian product.
We refer to Richter-Gebert \cite{RG} for an introduction to the projective plane and projective transformations,
and for the following facts.
The affine plane $\FF^2$ can be extended to a projective plane by adding a line $\lambda$ at infinity.
There are two points $\alpha, \beta$ on the line at infinity such that all lines through $\alpha$ (except for $\lambda$) are horizontal lines in the affine plane, and the lines through $\beta$ (except for $\lambda$ are vertical lines in the affine plane.
Projective transformations are those bijections of the projective plane that preserve collinearities and point-line incidences,
and we call two sets projectively equivalent if there is a projective transformation that maps one bijectively to the other.
For any two points $p,q$ there is a projective transformation that sends $p$ and $q$ to $\alpha$ and $\beta$ (see for instance \cite[Theorem 3.4]{RG}).

For each $i$, we let $\tau_i$ be a projective transformation sending $p_i$ and $q_i$ to $\alpha$ and $\beta$.
The preimage of the line at infinity is then the line $\overline{p_iq_i}$,
and from Lemma  \ref{lem:twopencils} we have $G_i\cap \overline{p_iq_i} = \emptyset$,
so $\tau_i$ maps $G_i$ into the affine plane.
Also, if $\overline{p_iq_i}$ happens to be in $\L$, 
then it has no incidences with $G_i$, so we can ignore it when bounding $\I(G_i,\L)$.
The set $H_i=\tau(G_i)\subseteq \FF^2$ is covered by $2^{15}K$ horizontal lines and $2^{15}K$ vertical lines,
so it is contained in a Cartesian product $X_i\times Y_i$ with $|X_i|=|Y_i|\leq 2^{15} K$. 
Since projective transformations preserve incidences, 
we have $\I(H_i,\L)=\I(G_i,\L)$.

We apply Theorem~\ref{thm:cartszemtrot} to bound the incidences on each product $X_i\times Y_i$.
In positive characteristic, the extra condition of Theorem~\ref{thm:cartszemtrot} holds, since the assumption $m^{-2}n^{13} \less p^{15}$ gives
\[|X_i||\L| \less Kn \less m^{-4/15}n^{26/15}\less p^2.\]
Therefore, we can apply Theorem~\ref{thm:cartszemtrot} to obtain 
(letting $c^*$ denote the implicit constant in Theorem~\ref{thm:cartszemtrot})
\begin{equation*}
\I(G_i,\L)
\leq \I(X_i\times Y_i, \L)
\leq c^*(2^{15}K)^{3/4}(2^{15}K)^{1/2}n^{3/4}
< c^*2^{20}K^{5/4}n^{3/4}.
\end{equation*}
Thus, using \eqref{eq:IlessIAL}, we have (recalling that $K = I/m$ and that $s\leq 2^{68}n^2/K^4$)
\[I 
\leq 2 \sum_{i=1}^s \I(G_i, \L)
< 2\cdot 2^{68} \frac{n^2}{K^4}\cdot c^*2^{20}K^{5/4}n^{3/4}
= 2^{89}c^* \frac{m^{11/4}n^{11/4}}{I^{11/4}}
.\]
Solving for $I$ gives  $I< C'm^{11/15}n^{11/15}$, 
for a constant $C'$ that depends only on the constant $c^*$ from Theorem~\ref{thm:cartszemtrot}, 
and not on $C$. 
Hence choosing $C>C'$ gives a contradiction to $I = Cm^{11/15}n^{11/15}$.
This concludes the proof of Theorem~\ref{thm:genszemtrot}.

\section{Applications}\label{sec:apps}

In this section we give a few corollaries of Theorems \ref{thm:genszemtrot} and \ref{thm:cartszemtrot}.
These are meant to give an impression of the possible applications, and we are certain there are more.
None of the proofs in this section are new, 
but we include them here to make it easy for the reader to verify the resulting exponents, as well as the extra condition in positive characteristic.
We briefly introduce each problem, but refer to the relevant papers for a more detailed background.

\subsection{Sum-product-type bounds}\label{sec:appscart}
As a first application, we reproduce the best known sum-product bound, which was first proved by Roche-Newton, Rudnev and Shkredov \cite{RRS}, also using Theorem \ref{thm:rudnev}.
Here we show that it follows from Theorem \ref{thm:cartszemtrot} using the same argument that Elekes \cite{Elekes} used to derive the sum-product bound \eqref{eq:elekes} over $\R$ from the Szemer\'edi-Trotter theorem.

\begin{corollary}\label{cor:sumprod}
Let $A\subset \FF$ be a finite set. 
If $\FF$ has positive characteristic $p$, then assume $|A|\less p^{5/8}$.
Then
\[\max\{|A+A|,|A\cdot A|\} \more |A|^{6/5}. \]
Moreover, if one of $|A+A|, |A\cdot A|$ is $O(|A|)$, 
then the other is $\Omega(|A|^{3/2})$.
\end{corollary}
\begin{proof}
Set $M_{max}:=\max\{|A+A|,|A\cdot A|\}$ and $M_{min}:=\min\{|A+A|,|A\cdot A|\}$.
Define a point set and line set by
\[\P:= (A+A)\times (A\cdot A),~~~~
\L := \{y = a'(x-a): (a,a')\in A\times A\}. \]

If $\FF$ has positive characteristic, we need to verify the condition $M_{min}|A|^2\less p^2$ of Theorem \ref{thm:cartszemtrot}.
Either $M_{min}\more |A|^{6/5}$, and we are done,
or $M_{min}\less |A|^{6/5}$, so that $M_{min}|A|^2\less p^2$ follows from the assumption $|A|\less p^{5/8}$.
The other condition of Theorem \ref{thm:cartszemtrot} is that $M_{min}M_{max}^2 \leq |\L|^3 = |A|^6$;
if this failed, it would imply $M_{max}\geq |A|^2$ and we would be done.

The line $y = a'(x-a)$ contains the point $(a''+a, a'a'')$ for any choice of $a''\in A$,
so each of the $|A|^2$ lines gives at least $|A|$ incidences.
Applying Theorem \ref{thm:cartszemtrot} gives
\[|A|^3 \leq I(\P,\L)\less M_{min}^{3/4}M_{max}^{1/2}|A|^{6/4},
\]
so
\begin{equation}\label{eq:minmax}
 M_{min}^{3}M_{max}^2 \more |A|^6,
\end{equation}
which implies the two statements in the corollary.
\end{proof}

The inequality $|A+A|^2|A\cdot A|^3\more |A|^6$ was obtained in \cite{RRS} with the condition $|A|\less p^{5/8}$, and $|A+A|^3|A\cdot A|^2\more |A|^6$ was obtained in \cite{AMRS} with the condition $|A|\less p^{3/5}$.
Equation \eqref{eq:minmax} combines both these inequalities, and improves the condition for the second one.

As a second application, we prove a lower bound on the size of the set $A\cdot (A+1)$,
a question raised by Bourgain \cite{Bourgain05}.
Our argument is again in the style of Elekes, 
and it was used over $\R$ by Garaev and Shen \cite{GS}.
Our bound is a slight improvement on a result of Rudnev, Shkredov and Stevens \cite{RSS}
(also based on Theorem \ref{thm:rudnev}), 
who proved the same bound up to logarithms.

\begin{corollary}\label{cor:aaplusone}
Let $A\subset \FF$ be a finite set.
If $\FF$ has positive characteristic $p$, then assume $|A|\less p^{5/8}$.
Then
\[|A\cdot (A+1)| \more |A|^{6/5}. \]
\end{corollary}
\begin{proof}
Define a point set and line set by
\[\P:= (A\cdot (A+1))\times (A\cdot (A+1)),
\L := \{y = a\cdot (x/(a'+1) +1 ): (a,a')\in A\times A\}. \]
Applying Theorem \ref{thm:cartszemtrot} gives the same calculation as in the proof of Corollary \ref{cor:sumprod},
with $M_{min}$ and $M_{max}$ replaced by $|A\cdot (A+1)|$.
\end{proof}

As another application of this kind, we consider the sets $A+BC$ and $A(B+C)$ for finite sets $A,B,C\subset \FF$.
Barak, Impagliazzo and Wigderson \cite{BIW} used \cite{BKT} to prove that there is an $\eps>0$ such that $|A+AA|\more |A|^{1+\eps}$ for every $A\subset \F$ with $|A|<p^{0.99}$.
Roche-Newton, Rudnev and Shkredov \cite{RRS} proved the bound 
\[|A+BC|\more \min\{(|A||B||C|)^{1/2}, M^{-1}|A||B||C|,p\}\]
 for $A,B,C\subset \F$, where $M = \max\{|A|,|B|,|C|\}$.
Aksoy Yazici et al. \cite{AMRS}  proved the same bound for $A(B+C)$.
Here we reprove both bounds, 
and we refine them somewhat by showing that the second term can be omitted as long as none of the sets is $\{0\}$ (if, say, $B=\{0\}$, then $|A+BC|\more (|A||B||C|)^{1/2}$ could not be true for large $C$ and small $A$).

\begin{corollary}\label{cor:aplusbc}
Let $A,B,C\subset \FF$ be finite sets, none of which equals $\{0\}$.
If $\FF$ has positive characteristic $p$, assume $|A||B||C|\less p^2$.
Then
\[|A+BC| \more (|A||B||C|)^{1/2} \hspace{10pt}\text{and}\hspace{10pt}|A(B+C)| \more (|A||B||C|)^{1/2}. \]
\end{corollary}
\begin{proof}
Note that we can assume $|B|\geq |C|$ by interchanging $B$ and $C$ if necessary.
Define a point set and line set by
\[\P:= C\times (A+BC),~~~~
\L := \{y = a+bx: (a,b)\in A\times B\}. \]
Each of the $|A||B|$ lines of $\L$ contains exactly $|C|$ points of $\P$, so there are $|A||B||C|$ incidences between $\P$ and $\L$.

In positive characteristic $p$, 
the condition $\min\{|C|,|A+BC|\}\cdot |\L|\less p^2$ of Theorem \ref{thm:cartszemtrot} holds because of the assumption $|A||B||C|\less p^2$.
The other condition of Theorem \ref{thm:cartszemtrot} is that $|C||A+BC|^2\leq (|A||B|)^3$,
which we may assume,
since otherwise we directly obtain
$|A+BC|^2> (|A||B|)^3|C|^{-1} \geq |A||B||C|$ using $|B|\geq |C|$.
Thus we can apply Theorem \ref{thm:cartszemtrot} to get
\[|A||B||C| =\I(\P,\L)
\less |C|^{3/4}|A+BC|^{1/2}(|A||B|)^{3/4}+|A||B|.\]
If the first term dominates, 
rearranging gives the first inequality of the corollary.
If the second term dominates,
we have $|C|=O(1)$.
Since $C\neq \{0\}$, we can pick a nonzero $c\in C$,
and observe that $|A+cB|\geq \max\{|A|,|B|\} \more (|A||B||C|)^{1/2}$.
This finishes the proof of the first inequality.

For the second inequality, we first remove $0$ from $A$, which does not affect the asymptotic behaviour (given that $A\neq \{0\}$).
Then we define
\[\P:= C\times (A(B+C)),~~~~
\L := \{y = a(b+x): (a,b)\in A\times B\}, \]
noting that the lines are distinct because $0\not\in A$.
The remaining calculation mirrors that of the first part.
\end{proof}

Remarkably, these bounds on $|A+BC|$ and $|A(B+C)|$ match the best known bounds over $\R$, obtained using the Szemer\'edi-Trotter Theorem
(see \cite[Exercise 8.3.3]{TV} or \cite[p. 287]{AS}).
When $A=B=C\subset \R$, 
Murphy et al. \cite{MRNS} managed to prove $|A(A+A)| \more |A|^{3/2+c}$ for a small $c>0$ 
(later improved by Roche-Newton \cite{RocheNewton}).

\subsection{An expanding polynomial}\label{sec:exppoly}
As another application of Theorem \ref{thm:cartszemtrot}, we prove an explicit expansion bound for the polynomial $f(x,y) = x^2+xy$.
This problem was first considered by Bourgain \cite{Bourgain05}, 
who used the result of \cite{BKT} to prove the following.
For all $0<\eps<1$ there is $\delta>0$ such that if $A,B\subset \F$ have size $|A|=|B|=N\approx p^\eps$, 
then $|f(A,B)| \more N^{1+\delta}$, 
where $f(A,B):=\{f(a,b):(a,b)\in A\times B\}$.
In other words, $f$ is a \emph{two-variable expander}.
Corollary \ref{cor:aplusbc} shows that $x+yz$ and $x(y+z)$ are three-variable expanders, but establishing expansion for two-variable polynomials appears to be harder.
Note that, in spite of Corollary \ref{cor:aaplusone},
$g(x,y)=xy+x$ is not an expander in this sense, 
because for distinct sets $A,B$ of size $N$ we can have $|g(A,B)|\less N$.

As far as we know, no explicit exponents have been published for Bourgain's problem over finite fields.
Hegyv\'ari and Hennecart \cite{HH} and Shen \cite{Shen} generalised Bourgain's bound to polynomials of a similar form.
Over $\R$, it is known that $|f(A,B)|\more N^{4/3}$ for $A, B\subset \R$ with $|A|=|B|=N$, 
and there is a general theory of which polynomials are expanders (see Raz, Sharir, and Solymosi \cite{RazSS}).
We prove the explicit expansion bound $|f(A,B)|\more N^{5/4}$ over any field,
using a proof similar to that of \cite[Theorem 4]{HH}.

\begin{corollary}
Consider the polynomial $f(x,y) = x^2+xy$, and finite sets $A,B\subset \FF$ with $A\neq \{0\}$. If $\FF$ has positive characteristic $p$, then assume $|A|^2|B|\less p^2$.
Then
\[|f(A,B)| \more \min\{|A|^{1/2}|B|^{3/4}, |B|^2\}. \]
\end{corollary}
\begin{proof}
Define
\[E := \{(a,b,a',b')\in A\times B\times A\times B : f(a,b) = f(a',b')\}. \]
By Lemma \ref{lem:cauchyschwarz}
(with $X = A\times B$, $Y = f(A,B)$, $Z=\FF$ and $\varphi(a,b) = f(a,b)$)
 we get
\begin{equation}\label{eq:csbound2}
|A||B| \leq |f(A,B)|^{1/2}|E|^{1/2}. 
\end{equation}

On the other hand, we can bound $|E|$ using Theorem \ref{thm:cartszemtrot}, 
by viewing a solution of the equation $a^2+ab = (a')^2 + a'b'$ as an incidence between the point $(b,b')$ and the line $a^2+ax = (a')^2+a'y$. 
Define a point set and line set by
\[\P := B\times B, ~~~~~~\L := \left\{
ax -a'y = (a')^2-a^2
: (a,a')\in A\times A, a\neq\pm a'\right\}. \]
The lines in $\L$ are distinct (because of the restriction $a\neq\pm a'$), and $\I(\P,\L)\geq |E|/2$.

The condition $|B||\L|\less p^2$ of Theorem \ref{thm:cartszemtrot} follows directly from the assumption $|A|^2|B|\less p^2$. 
The other condition of Theorem \ref{thm:cartszemtrot} is that $|B|^3\leq |\L|^3=|A|^6$.
If this fails, then $|B|>|A|^2$ gives $|f(A,B)| \geq |B| > |A|^{1/2}|B|^{3/4}$, where the first inequality is obtained by considering the values of $f(x,y)$ with $x$ any fixed nonzero element of $A$ (using the assumption $A\neq \{0\}$).
So we can apply Theorem \ref{thm:cartszemtrot} to get
\[ |E| \leq 2\I(\P,\L)\less |B|^{3/4}|B|^{1/2}(|A|^2)^{3/4} +|A|^2 = |A|^{3/2}|B|^{5/4}+|A|^2.\]
Together with \eqref{eq:csbound2} this gives the bound in the corollary.
\end{proof}

\subsection{Distinct distances}
As mentioned in Section \ref{sec:intro}, Guth and Katz \cite{GK} solved the distinct distance problem in $\R^2$ (up to a logarithmic factor). 
Write $d(q,r) = (q_x-r_x)^2 + (q_y-r_y)^2$ for the squared Euclidean distance between two points $q$ and $r$, 
and write $\Delta(\P) = \{d(q,r): q,r\in \P\}$ for the set of distances determined by $\P$.
Guth and Katz proved that for $\P\subset \R^2$ we have $|\Delta(\P)| \more |\P|/\log|\P|$.

A related problem is the `pinned distance' problem, which asks for the existence of a point from which many distinct distances occur.
We write $\Delta_q(\P) = \{d(q,r): r\in \P\}$ for the set of distances `pinned' at $q$.
The approach of \cite{GK} does not apply to this variant, and the best known bound is due to Katz and Tardos \cite{KT}, 
who proved that for any $\P\subset \R^2$, there is a $q\in \P$ such that $|\Delta_q(\P)| \more |\P|^{0.86}$.

The finite field version of this problem was first considered by Bourgain, Katz and Tao \cite{BKT}, 
who proved that for $\P\subset \F^2$, with $|\P|= p^\alpha$ and $0<\alpha<2$, there is a $q\in \P$ such that $|\Delta_q(P)|\more |\P|^{1/2+\eps}$ for some $\eps = \eps(\alpha)>0$. 
For large $\alpha$, 
explicit versions of this statement are known,
with the current best due to Hanson, Lund, and Roche-Newton \cite{HLR},
who proved that $|\P|\geq p^{4/3}$ implies that there is a $q\in \P$ with $|\Delta_q(\P)| \more p$.
As far as we know, for $\alpha <4/3$ no explicit values have been published for $\eps$ in the statement of Bourgain, Katz and Tao.
Here we prove that for $\alpha \leq 15/11$ we can take $\eps = 1/30$ (but we note that for $\alpha \geq 4/3$ this is weaker than the bound of \cite{HLR}).
Our proof is essentially that of \cite{BKT}, but we take some more care to deal with the case where $-1$ is a square in $\FF$.

To avoid degeneracies, \cite{BKT} proved their theorem only for finite fields in which $-1$ is not a square\footnote{This is not stated in the journal version of \cite{BKT}, but it is mentioned in Section  7 of the later version {\tt arXiv:math/0301343v3}.}. Indeed, if $-1$ is a square in $\FF$, then the plane $\FF^2$ has \emph{isotropic lines}.
The defining property of an isotropic line is that the distance between any two points on the line is zero. 
Explicitly, for $r=(r_x,r_y)\in \FF^2$,
there are two isotropic lines passing through $r$,
defined via the equations $(y-r_y)=\pm \imath\cdot  (x-r_y)$, where $\imath^2=-1$. 
We use $\lambda_r$ and $\mu_r$ to denote the isotropic lines of $r$.
A point set $\P$ contained in an isotropic line has $\Delta(\P) = \{0\}$;
we exclude this case explicitly.

\begin{corollary}\label{cor:distances}
Let $\P$ be a set of $m$ points in $\FF^2$.
If $\FF$ has positive characteristic, assume 
$m\less p^{15/11}$.
If $-1$ is a square in $\FF$, assume that $\Delta(\P) \neq \{0\}$.
Then there exists a point $q\in \P$ such that
\[|\Delta_q(\P)| \more m^{8/15}. \]
In particular, we have $|\Delta(\P)| \more m^{8/15}$.
\end{corollary}
\begin{proof}
For two distinct points $r,s\in \FF^2$, the set
\[\ell_{rs} := \{q\in \FF^2:  d(q,r) = d(q,s) \} \]
is a line (over $\R$, this is the \emph{perpendicular bisector} of $r$ and $s$).
If $-1$ is not a square in $\FF$, then for a fixed point $r$, distinct $s$ give distinct lines.
However, when $-1$ is a square in $\FF$, then for any two points $s,t$ on one of the isotropic lines $\lambda_r,\mu_r$ we have $\ell_{rs} = \ell_{rt}$, which would cause a problem in the counting argument below.

To deal with these isotropic lines, first observe that if any line contains at least $m/3$ 
points of $\P$, but $\P$ has a point $q$ outside that line, then $\Delta_q(\P)\more m$.
By the assumption that $\Delta(\P)\neq \{0\}$, we know that $\P$ is not contained in an isotropic line, so we can assume that any isotropic line contains at most $m/3$ points of $\P$.
For any fixed $r\in \P$, this implies that there are at least $m/3$ points not on $\lambda_r$ or $\mu_r$.
We will use this observation below.

For a fixed $r\in \P$, 
define the set of lines
\[ \L_r := \{ \ell_{rs}: s\in \P, d(r,s)\neq 0\}.   \]
We have $|\L_r|\leq m$,
so Theorem \ref{thm:genszemtrot} gives
\begin{equation*}\label{eq:distancesincs}
\I(\P,\L_r) \less m^{22/15}. 
\end{equation*}

Distinct $s$ with $d(r,s)\neq 0$ give distinct lines in $\L_r$.
To see this, assume (by translating the point set) that $r=(0,0)$,
so that the line $\ell_{rs}$ is  given by  $2s_xx +2s_yy = s_x^2+s_y^2$.
This implies that distinct $s$ give distinct lines, unless $s_x^2+s_y^2=0$,
which is excluded in the definition of $\L_r$.
It follows that an incidence $q\in \ell_{rs}$ corresponds to exactly two triples $(q,r,s),(q,s,r)\in \P^3$ such that $d(q,r)= d(q,s)\neq 0$ (over $\R$, these triples form \emph{isosceles triangles}, at least when the three points are not collinear).
Thus we have
\begin{align*}
\sum_{r\in \P}\I(\P,\L_r) &\approx |\{(q,r,s)\in \P^3 : d(q,r) = d(q,s)\neq 0\}|\\
&= \sum_{q\in \P} |\{(r,s)\in \P^2: d(q,r) = d(q,s)\neq 0\}|.
\end{align*}

For every $q\in \P$, Lemma \ref{lem:cauchyschwarz} (with $X=\P\backslash(\lambda_q\cup\mu_q)$, $Y = \Delta_q(\P)\backslash\{0\}$, $Z = \FF\backslash\{0\}$ and $\varphi(r) = d(q,r)$)
gives:
\begin{align*}
|\{(r,s)\in \P^2: d(q,r) &= d(q,s)\neq 0\}| \geq  \frac{|\{(r,\delta)\in \P\times \Delta_q(\P): d(q,r) = \delta\neq 0\}|^2}{|\Delta_q(\P)|}.
\end{align*}
As we assumed that at least $m/3$ points of $\P$ are not on an isotropic line through any given point in $\P$, we have, for every $q\in \P$,
\[|\{(r,\delta)\in \P\times \Delta_q(\P): d(q,r) = \delta\neq 0\}| \geq m/3.\]
Hence we may conclude that
\[m^{37/15}\more \sum_{r\in \P}\I(\P,\L_r)
\more  \sum _{q\in \P} \frac{m^2}{ |\Delta_q(\P)|}
 \geq  \frac{m^3}{\max_q|\Delta_q(\P)|}.\]
This gives $\max_q |\Delta_q(\P)| \more m^{8/15}$.
\end{proof}

For $|A|\less p^{2/3}$, using Theorem \ref{thm:cartszemtrot} in the proof above gives $|\Delta_q(A\times A)|\more  |A|^{5/4}$, improving on \cite[Corollary 13$(a)$]{AMRS}.
Petridis \cite{Petridis} improved this to $|\Delta_q(A\times A)|\more |A|^{3/2}$, 
but his proof does not seem to apply to unstructured point sets as in Corollary \ref{cor:distances}.

\subsection{Beck's theorem}

As a final application,
we show that Theorem~\ref{thm:genszemtrot} leads to a sharper bound over general fields for a theorem of Beck \cite{Beck}, also known as `Beck's theorem of two extremes'.
We say that a line $\ell$ is \emph{determined} by a point set $\P$ if $\ell$ contains at least two points of $\P$. 
Beck proved that for any set $\P$ of $m$ points in $\R^2$, either $(i)$ $\P$ has $\Omega(m)$ points on a line, or $(ii)$ $\P$ determines $\Omega(m^2)$ distinct lines.

Naturally there are similar results in other settings. 
In $\mathbb{F}_p^2$, Helfgott and Rudnev \cite{HR} established that if $m<p$ and $\P=A\times A$ (so no line has $\Omega(m)$ points),
then $\P$ determines $\Omega(m^{1+1/267})$ lines.
Jones \cite{Jones12} removed the Cartesian product condition, 
proving that either $\P$ has $\Omega(m)$ points on a line,
or $\P$ determines $\Omega(m^{1+1/109})$ lines.
As in the remark directly after Theorem \ref{thm:tj},
Jones's argument can be improved using \cite[Theorem 6]{RRS};
the resulting exponent would be $1 +1/53$.
For large point sets,  
Alon \cite{Alon} proved that any point set $\P\subset \FF_q^2$ of size $m> q$ determines $cm^2$ lines, with $c$ depending on $m/q$.

Aksoy Yazici et al. \cite{AMRS} used Theorem~\ref{thm:rudnev} to show that $\P=A\times A\subseteq \FF^2$ determines $\Omega(m^{3/2})$ lines over an arbitrary field $\FF$ (assuming $m\less p^{4/3}$ in positive characteristic).
We present a more general, albeit weaker, result for an unstructured point set, improving on \cite{Jones12}. 
We deduce it from the incidence bound in Theorem~\ref{thm:genszemtrot}
using a standard argument.
Using Theorem \ref{thm:cartszemtrot} instead of Theorem \ref{thm:genszemtrot} in our proof of Corollary \ref{cor:beck} would give the same result as in \cite{AMRS}.
Our proof works for $m\less p^{7/6}$, but we note that for $m>p$ it is weaker than the result of Alon \cite{Alon}.

\begin{corollary}\label{cor:beck}
Let $\P$ be a set of $m$ points in $\FF^2$. 
If $\FF$ has positive characteristic $p$, suppose that $m\less p^{7/6}$. 
Then one of the following is true:
\begin{enumerate}[label=(\roman*)]
\item[$(i)$] $\P$ has $\Omega(m)$ points on a line;
\item[$(ii)$] $\P$ determines $\Omega(m^{8/7})$ lines.
\end{enumerate}
\end{corollary}
\begin{proof}
Let $\L$ be the set of lines determined by $\P$.
Partition $\L$ into $\lfloor\log_2 m\rfloor$ sets $\L_j\subseteq \L$ such that 
$$\L_j =\{ \ell\in \L:2^j\leq |\ell\cap \P|<2^{j+1}\}.$$

Consider $\L_j$ with $|\L_j|>m^{7/8}$.
We can assume that each $|\L_j|< m^{8/7}$, since otherwise $(ii)$ holds.
Thus, in positive characteristic, we can use the assumption $m\less p^{7/6}$ to get $m^{-2}|\L_j|^{13}\less p^{15}$.
This lets us apply Theorem \ref{thm:genszemtrot} to $\P$ and $\L_j$.
Since every line in $\L_j$ gives at least $2^j$ incidences, Theorem \ref{thm:genszemtrot} gives
\[2^j|\L_j| \less m^{11/15} |\L_j|^{11/15}, \]
which leads to 
\[ |\L_j|\less   \frac{m^{11/4}}{(2^j)^{15/4}}. \]
Each line in $\L_j$ contains $\Theta(2^{2j})$ pairs of points, 
so all the lines in $\L_j$ together contain $O( m^{11/4} (2^j)^{-7/4})$ pairs of points.

For $\L_j$ with $m^{1/2}\leq|\L_j|\leq m^{7/8}$,
Lemma \ref{lem:combinatorial} gives 
$2^j|\L_j|\less m^{1/2}|\L_j|$,
so $2^j\less m^{1/2}$,
and then the number of pairs of points on lines of $\L_j$ is $O(m^{15/8})$.
Finally,
if $|\L_j|\leq m^{1/2}$, then Lemma \ref{lem:combinatorial} gives 
$2^j|\L_j|\less m$, so $|\L_j|\less 2^{-j}m$,
and the number of pairs of points on lines of $\L_j$ is $O(2^jm)$.

Let $C$ be a large constant and let $U$ be the union of all $\L_j$ for $Cm^{3/7}\leq 2^j\leq m/C$. 
By the estimates above, there are then at most 
$$O\left(\frac{m^{11/4}}{ (Cm^{3/7})^{7/4}}+m^{15/8}\log m+ \frac{m}{C}\cdot m\right)=O\left(\frac{m^2}{C}\right)$$
pairs of points on lines in $U$. 

For sufficiently large $C$, this quantity is less than $\frac{1}{2}\binom{m}{2}$.
Thus, the remaining $\Omega(m^2)$ pairs of distinct  points of $\P$ must lie outside $U$.
Either a positive proportion of these pairs are supported on lines containing more than $m/C$ points, or a positive proportion of pairs lie on lines with less than $Cm^{3/7}$ (and at least two) points. 
So either there is a line containing at least $m/C$ points,
and $(i)$ holds, 
or there are 
\[\Omega\left( \frac{m^2}{(Cm^{3/7})^2} \right) = \Omega\left( m^{8/7}\right)\]
distinct lines defined by pairs of points of $\P$,
and $(ii)$ holds.
This completes the proof of the corollary.
\end{proof}

\section*{Acknowledgements} 
The first author is grateful to Misha Rudnev for his guidance, and for numerous invaluable discussions and suggestions.
The second author also thanks Misha Rudnev for some helpful conversations.
Both authors would like to thank Mark Lewko, Brendan Murphy, Thang Pham, Ilya Shkredov and an anonymous referee for valuable comments.
The second author was partially supported by Swiss National Science Foundation grants 200020--165977 and 200021--162884.

\end{document}